\documentclass[a4paper,12pt]{article}
\usepackage[all]{xy} 
\usepackage{tikz-cd}

\usepackage{amssymb}
\usepackage{epsfig}
\usepackage{amsfonts}
\usepackage{amsmath}
\usepackage{euscript}
\usepackage{amscd}
\usepackage{amsthm}
\DeclareMathAlphabet{\mathpzc}{OT1}{pzc}{m}{it}

\newtheorem{thm}{Theorem}[section]
\newtheorem{lem}[thm]{Lemma}
\newtheorem{prop}[thm]{Proposition} 
\newtheorem{cor}[thm]{Corollary}
\newtheorem{rem}[thm]{Remark}
\newtheorem{ex}[thm]{Example}

\newcommand{\rar}{\rightarrow}
\newcommand{\m}{\mathpzc{m}}
\newcommand{\n}{\mathpzc{n}}
\newcommand{\p}{\mathpzc{p}}
\newcommand{\q}{\mathpzc{q}}

\newcommand{\bQ}{\mathbb Q}
\newcommand{\bR}{\mathbb R}
\newcommand{\bC}{\mathbb C}
\newcommand{\bN}{\mathbb N}

\newcommand{\A}{\mathbb A}

\newcommand{\Exp}{\operatorname{Exp}}

\newcommand{\hgt}{\operatorname{ht}}

\newcommand{\td}{\operatorname{tr.deg}}

\title{Some results on retracts of polynomial rings}
\author{Sagnik Chakraborty$^*$, Nikhilesh Dasgupta$^{**}$, Amartya Kumar Dutta$^{\$}$ \\
and Neena Gupta$^\dagger$\\
{\small{\it $^*$Department of Mathematics, Ramakrishna Mission Vivekananda Educational and Research Institute,}}\\
{\small{\it G. T. Road, P.O. Belur Math, Howrah, Kolkata - 711 202, India.}}\\
{\small{\it e-mail: jusagnik28@gmail.com}}\\
{\small{\it $^{**}$, $^{\$}$ and $^\dagger$ Stat-Math Unit, Indian Statistical Institute,}}\\
{\small{\it 203 B.T. Road, Kolkata 700 108, India.}}\\
{\small{\it e-mail : $^{**}$its.nikhilesh@gmail.com, $^{\$}$amartya.28@gmail.com, $^\dagger$ neenag@isical.ac.in}}\\
}

\begin{document}
\date{}
\maketitle
\abstract{In this paper, we first consider the relationship between a polynomial ring $B$ 
over a Noetherian domain $R$ and the ring of invariants $A$ of a ${\mathbb G}_a$-action on $B$, when $A$ occurs as a retract of $B$.
Next, we study  retracts of a polynomial ring in general and 
address the questions of D. L. Costa raised in \cite{C}. 
Finally, we examine the behaviour of ideals and certain properties of rings under retractions.
\smallskip

\noindent
{\small {{\bf Keywords}. Retract, polynomial ring, locally nilpotent derivation, 
${\mathbb G}_a$-action, exponential map, $\A^2$-fibration.}

}}

\section{Introduction}
Let $R\subseteq A \subseteq B$ be commutative rings. 
The ring $A$ is said to be an $R$-algebra retract of $B$
if there exists an $R$-algebra homomorphism $\pi: B\to B$
such that $\pi^2=\pi$ and $\pi(B)= A$. 

Now suppose $R$ is a Noetherian domain containing $\bQ$, 
$B:= R[X_1,X_2, \dots, X_n]$, a polynomial ring in $n$ variables over $R$,
and $A$ is the kernel of a non-zero locally nilpotent derivation $D$ on $B$.
It is well-known that $B_a=A_a[T]$ for some  $a (\ne 0) \in A$ and $T \in B$ transcendental over A (Lemma \ref{lnd}(iii)).
One investigates the structure of $B$ over $A$ and conditions under which $B$ itself is isomorphic to the polynomial algebra over $A$.
For instance, in \cite{BD2}, it is shown that when $R$ is a Noetherian domain containing $\bQ$, $B=R[X_1,X_2]$ and $(DX_1, DX_2)B=B$
then $B=A[T]$ for some $T \in B$.  In this paper, we investigate the above problem when the kernel $A$ occurs as a retract of $B$.
More generally, we consider the ring of invariants $A$ of any exponential map $\phi$ 
(the ring-theoretic version of a ${\mathbb G}_a$-action, defined in Section 2)
on the polynomial ring $B$ over any Noetherian domain $R$.
We prove (Theorem \ref{rlnd1}) that {\it when $R$ is a Noetherian normal domain  and 
 $ B=R[X_1, \dots,X_n]$, then $B$, as an  $A$-algebra, is  isomorphic to the symmetric algebra of $IA$ for some invertible ideal $I$ of $R$}.
As a step to Theorem \ref{rlnd1}, we first show (Proposition \ref{ufdp}) that {\it when $R$ is a UFD then $B= A[T]$ for some $T \in A$}.
For the convenience of readers who are more comfortable with the language of locally nilpotent derivations,
the corresponding results for the latter have also been stated separately (Corollaries \ref{ufdpc} and \ref{rlnd1c}).
The proof of Theorem \ref{rlnd1} involves a new result (Lemma \ref{patch}) on the concept of $\A^1$-patch that was formally defined in 
\cite[Definition 2.1]{DGO}. This result is a variant of the patching result \cite[Lemma 3.1]{BD2} of Bhatwadekar-Dutta.

\smallskip

In this paper we also revisit the questions of D. L. Costa (\cite[Section 4]{C}) on 
$R$-algebra retracts of polynomial rings $R[X_1, \dots, X_n]$ (see Section 5). 
We show  that {\it any retract  of $R[X_1, \dots, X_n]$ of transcendence degree one over a Noetherian domain $R$ is
an $\A^1$-fibration over $R$} (Theorem \ref{ctd1}) and that {\it any retract of $R[X_1, \dots, X_n]$ of transcendence degree 
two over a Noetherian domain $R$ containing $\bQ$ is an $\A^2$-fibration over $R$} (Theorem \ref{td2n}). 
We shall discuss the implications of these results in the light of some well-known results and examples.

We observe that {\it if $k$ is a field of characteristic zero, then any retract of 
$k[X_1, X_2, X_3]$ is a polynomial ring} (Theorem \ref{ftd2}). This result follows from a characterization of polynomial
subrings of $k[X_1, X_2, X_3]$ by Miyanishi, Sugie and Fujita (Theorem \ref{r}) and has been independently observed by T. Nagamine (\cite{T}). 
On the other hand, over any field $k$ of positive characteristic,
the counterexamples to the Zariski Cancellation Problem by the fourth author in \cite{G} and \cite{G2} show that
when $n \ge 4$, retracts of $k[X_1,\dots,X_n]$ need not be polynomial rings.

We shall prove that over any field $k$, {\it a retract $A$ of $B:=k[X_1,\dots,X_n]$ 
is again a polynomial ring over $k$, provided $A$ is a graded subring 
of $B$ and the irrelevant maximal ideal of $B$ remains invariant under the retraction} 
(see Theorem \ref{gradedr} for a more general statement over an integral domain $R$). 

The paper also records some general results on retracts in the spirit of Costa's results in \cite{C}, 
including results on properties of rings and ideals preserved under retractions. Some of these results give conditions
for a ring to be faithfully flat over its retract. 
One of the results (Theorem \ref{complocalt}) gives an analogue of Theorem \ref{gradedr} for retracts of a complete 
equicharacteristic regular local ring.

We now discuss the layout of the paper. In Section 2, we set up the notation and recall a few definitions and  known results. 
In Section 3, we prove a new result on $\A^1$-patch and 
in Section 4, we study the ring of invariants of a ${\mathbb G}_a$-action on a polynomial ring which is also a retract
of the polynomial ring.
In Section 5, we discuss the Questions of Costa and
in Section 6, we record a few miscellaneous results on retracts.

\section{Preliminaries}

{\bf \textsl{Notation:}}

By a ring, we will mean a commutative ring with unity. We denote the group of units of a ring $R$ by $R^{*}$. 
For a ring $R$ and a non-zerodivisor $f \in R$, we use $R_f$ to denote the localisation of $R$ with respect 
to the multiplicatively closed set $\{1,f,f^2,\dots\}$. We denote the field of fractions of an integral domain $R$ by 
    $Q(R)$. {\bf  The notation $k$ will always denote a field}.
     
\smallskip
\noindent

Let $A \subseteq B$ be integral domains. Then the 
transcendence degree of the field of fractions of $B$ over that of $A$ is denoted by $\td_{A}{B}$. 
For a ring $R$ and a prime ideal $\p$ of $R$, $\kappa(\p)$ denotes the 
residue field of the local ring $R_{\p}$; and if $A$ is an $R$-algebra, we use $A_\p$ to denote the ring
$S^{-1}A$, where $S:= R\setminus\p$. For an $R$-module $M$, ${\rm Sym}_R(M)$ denotes the symmetric algebra of $M$.
\smallskip
\noindent

An ${\bN}$-{\it graded ring} $R$ is a ring together with a 
direct sum decomposition of $R$ as an additive group $R=\bigoplus_{i \in {\bN}}R_{i}$ such that  
$R_{i}R_{j} \subseteq R_{i+j}$ for all $i,j$. A non-zero element $r \in R$ is said to be {\it homogeneous} if $r \in R_{i}$ 
for some $i \in {\bN}$ and $i$ is called the {\it degree} of $r$. 
The ideal of $R$ generated by the homogeneous elements of positive degree is 
called the {\it irrelevant ideal} and is denoted by $R_{+}$. Note that $R=R_0\oplus R_{+}$.

\smallskip
\noindent

Let $R$ be a ring and $n$ a positive integer. For an $R$-algebra $A$, we use the notation $A=R^{[n]}$ to denote that $A$ is 
isomorphic to a polynomial ring in $n$ variables over $R$ and the notation $A=R^{[[n]]}$ to denote that $A$ is isomorphic to 
a power series ring in $n$ indeterminates over $R$. 
\medskip

\noindent
{\bf \textsl{Definitions:}}

A subring $A$ of a ring $B$ is said to be a {\it retract} of $B$ if there exists an idempotent endomorphism 
$\pi : B \rightarrow B$ such that $\pi(B)=A$. The map $\pi$ is called a {\it retraction}.

\smallskip

A finitely generated flat $R$-algebra $A$ is called an {\it ${\A}^{n}$-fibration over $R$} if,
for each prime ideal $\p$ of $R$, $ A \otimes_{R} \kappa(\p) ={\kappa(\p)}^{[n]}  $.

\smallskip

A derivation $D$ on a ring $B$ is said to be {\it locally nilpotent} if, for each $b \in B$, there exists a positive integer 
$n$ (depending on $b$) such that $D^{n}(b)=0$. When $B$ is an $R$-algebra, we denote the set of locally nilpotent 
$R$-derivations of $B$ by $LND_{R}(B)$. The {\it kernel} of a locally nilpotent derivation $D$ is denoted by $Ker ~ D$.

\smallskip

Let $R$ be a ring and $\phi: B \to B^{[1]}$ be an $R$-algebra homomorphism. For an indeterminate $U$ over $B$, let 
$\phi_U$ denote the map $\phi: B \to B[U]$.  Then
$\phi$ is said to be an { \it exponential map} on $B$ if $\phi$ satisfies the following two properties:
\begin{enumerate}
\item [\rm (i)] $\varepsilon_0 \phi_U$ is identity on $B$, where 
$\varepsilon_0: B[U] \to B$ is the evaluation at $U = 0$.
\item[\rm (ii)] $\phi_V \phi_U = \phi_{V+U}$, where 
$\phi_V: B \to B[V]$ is extended to a homomorphism 
$\phi_V: B[U] \to B[V,U]$ by  setting $\phi_V(U)= U$.
\end{enumerate}
We denote the ring of invariants of $\phi$, i.e., the subring $\{a \in B\,| \,\phi (b) = b\}$ of $B$ by $B^{\phi}$ and  
the set of all $R$-algebra exponential maps on $B$ by ${\rm Exp}_R(B)$. 

If $R$ contains $\bQ$, then any locally nilpotent derivation $D$ on $B$ gives
rise to an exponential map $\phi: B \to B[T]$ defined by  
$$
\phi:=\sum_{n \ge o}\frac{D^n}{n!}T^n
$$
and conversely, any exponential map of $B$ is of the above form for some locally nilpotent derivation.
For instance, if $B=k[x]=k^{[1]}$, then the derivation $\frac{\partial}{\partial x}$ induces the exponential map 
$\phi: B \to B[T]$ defined by $\phi (x)=x+T$.
\smallskip

A subring $A$ of $B$ is said to be {\it factorially closed} in 
$B$ if, for all $a,b \in B$, $ab \in A\setminus \{0\}$ implies $a,b \in A$.
\medskip

\noindent
{\bf \textsl{{Preliminary results:}}}

\smallskip
We first recall a few important properties of retracts recorded by Costa in \cite{C}.

\begin{lem}\label{ct}
Let $A$ be a subring of the integral domain $B$ such that $A$ is a retract of $B$. Then the following statements hold:
\begin{enumerate}
 \item [\rm (i)] If $B$ is an integral domain, then $A$ is algebraically closed in $B$ {\em (\cite[1.3]{C})}.
 \item[\rm (ii)]If $C$ is an $A$-algebra, then $C=A \otimes_{A} C$ is a retract of $B \otimes_{A} C$. In particular, if 
 $S$ is a multiplicatively closed subset of $A$ then $S^{-1}A$ is a retract of $S^{-1}B$; and if $Q$ is an ideal of $A$, then 
 $\frac{A}{Q}$ is a retract of $\frac{B}{QB}$ {\em (\cite[1.9]{C})}.
 \item[\rm (iii)] If $B$ is Noetherian then $A$ is Noetherian {\em (\cite[1.2]{C})}. 
If $B$ satisfies the ascending chain condition on principal ideals then so does $A$ {\em (\cite[1.8]{C})}.
 \item [\rm (iv)]If $B$ is a UFD, then so is $A$ {\em (\cite[1.8]{C})}.
 \item[\rm (v)]If $B$ is regular, then so is $A$ {\em (\cite[1.11]{C})}.   
\item[\rm (vi)]If $B$ is normal, then so is $A$ {\em (\cite[1.6]{C})}.
\end{enumerate}
\end{lem}

The following theorem characterizes retracts of polynomial rings over a UFD in some special cases 
(\cite[Theorem 3.5 and subsequent Remark]{C}).

\begin{thm}\label{ct2}
Let $R$ be a UFD and $A$ a retract of $B=R[X_{1}, X_{2}, \dots ,X_{n}](=R^{[n]})$. 
\begin{enumerate}
 \item [\rm (i)]If $\td_{R}A=0$, then $A=R$. 
 \item [\rm (ii)]If $\td_{R}A=1$, then $A=R^{[1]}$.
 \item [\rm (iii)]If $\td_{R}A=n$, then $A=B$.
 \end{enumerate}
\end{thm}

We now state an elementary result on symmetric algebras (\cite[Lemma 3]{EH}).

\begin{lem}\label{eh}
Let $R$ be a ring and $M,N$ be finitely generated $R$-modules. Then the following statements are equivalent:
\begin{enumerate}
\item[\rm(I)] $M \cong N$ as $R$-modules.
\item[\rm(II)] ${\rm Sym}_R(M) \cong {\rm Sym}_R (N)$ as $R$-algebras.
\end{enumerate} 
\end{lem}

The following local-global theorem was proved by H. Bass, E. H. Connell and D. L. Wright (\cite{BCW}), 
and independently by A. A. Suslin (\cite{SUS}).
\begin{thm}\label{bcw}
Let $R$ be a ring and $A$ a finitely presented $R$-algebra. Suppose that for each maximal 
ideal $\m$ of $R$, $A_{\m}= {R_{\m}}^{[n]}$ for some integer $n \ge 0$. 
Then $A \cong {\rm Sym}_R{(M)}$ for some finitely generated projective $R$-module $M$ of rank $n$. 
\end{thm}

Next we state a result of A. Sathaye (\cite[Theorem 1]{S}) on the triviality of ${\A}^{2}$-fibrations over a discrete 
valuation ring containing $\bQ$.

\begin{thm}\label{sth}
Let $R$ be a discrete valuation ring containing $\bQ$. If $A$ is an ${\A}^{2}$-fibration over $R$, then $A=R^{[2]}$.
\end{thm}

The following version of Russell-Sathaye criterion \cite[Theorem 2.3.1]{RS} and of
Dutta-Onoda result \cite[Theorem 2.4]{DO} was proved  by Das-Dutta in \cite[Proposition 2.5]{DD}.

\begin{thm}\label{rsdd}
Let $A$ be a subring of an integral domain $B$ with a retraction $\pi: B \to A$.
Suppose that there exists a prime element $p \in A$ such that 
\begin{enumerate}
  \item [\rm (i)] $p$ is prime in $B$.
 \item [\rm (ii)] $B[1/p]={A[1/p]}^{[1]}$.
 \item [\rm (iii)] $\bigcap_{n \geqslant0} p^nB=(0)$.  
\end{enumerate}
Then there exists an element $x \in B$ such that $B=A[x]=A^{[1]}$.
\end{thm}

Next we state the well-known cutting down lemma of Eakin (\cite[Lemma B]{E}).
\begin{lem}\label{e}
Let $k \subseteq A \subseteq k^{[n]}$ for some positive integer $n$. 
Then $A$ can be embedded inside $k^{[d]}$, where $d= \dim A \leqslant n$. 
\end{lem}

The following theorem was proved by Fujita (\cite{Fu}) and Miyanishi-Sugie (\cite{MS}) in characteristic zero, and 
by Russell (\cite[Theorem 3]{R}) in arbitrary characteristic.
\begin{thm}\label{r}
Let $k$ be a perfect field with an algebraic closure $\bar{k}$. Let $B= k^{[2]}$ and $A$ a finitely generated regular $k$-subalgebra 
of $B$ of dimension $2$ such that ${\bar{k}} \otimes_{k} A$ is a UFD and $Q(B)|_{Q(A)}$ is a separable extension. Then $A=k^{[2]}$.
\end{thm}

Next we state a useful criterion for flatness (\cite[20.G]{MA}).
\begin{lem}\label{akl}
Let $R \rightarrow A$ and $A \rightarrow B$ be local homomorphisms of Noetherian local rings. Suppose that $A$ is flat 
over $R$. Then $B$ is flat over $A$ if (and only if) the following two conditions hold:
\begin{enumerate}
 \item [\rm (i)]$B$ is flat over $R$.
 \item [\rm (ii)]$B \otimes_{R} L$ is flat over $A \otimes_{R} L $, where $L:=R/{\m}_{R}$.
\end{enumerate}
\end{lem}

For convenience, we state below a well-known result on flatness (\cite[Theorem 7.4(i)]{M}).
\begin{lem}\label{bourbaki}
Let $R$ be a ring, $M$ a flat $R$-module, $N$ an $R$-module and let $N_1,N_2$ be submodules of the $R$-module $N$.
Then, considering all the modules below as submodules of $N \otimes_R M$, we have
$$(N_1 \cap N_2) \otimes_R M= (N_1 \otimes_R M) \cap (N_2 \otimes_R M).   $$ 
\end{lem}

Finally, we recall some useful properties of exponential maps (\cite[pp. 1291-1292]{Cr}).
\begin{lem}\label{lnd}
Let $B$ be an integral domain, $\phi\in \Exp_{R}(B)$ and $A:=B^{\phi}$. 
Then the following statements hold:
\begin{enumerate}
 \item [\rm (i)] $A$ is a factorially closed subring of $B$. Consequently, $A$ is  algebraically closed in $B$ and
  if $B$ is a UFD, then $A$ is also a UFD.
 \item [\rm (ii)] For a multiplicatively closed set $S\subseteq A \setminus \{0\}$, $\phi$ extends to an exponential map  
 of $S^{-1}B$ with ring of invariants $S^{-1}A$ and $B \cap S^{-1}A=A$. 
 \item [\rm (iii)] If $\phi$ is non-trivial (i.e., $\phi \neq Id$), then there exists a non-zero element $a \in A$ such that 
 $B_{a}={A_{a}}^{[1]}$. In particular, $\td_{A}B=1$.
\end{enumerate}
\end{lem}

\section{A patching result}

In this section we prove a result on $\A^1$-patch (Lemma \ref{patch}).  
For convenience, we first state a few elementary results. 

\begin{lem}\label{ff}
Let  $R$ be a ring, $B$ an $R$-algebra and $A$ an $R$-algebra retract of $B$. 
Then the following statements hold:
\begin{enumerate}
 \item[\rm(i)] If $B$ is a finitely generated $R$-algebra, then $A$ is a finitely generated $R$-algebra.
\item[\rm(ii)] If $B$ is a faithfully flat $R$-algebra, then $A$ is a faithfully flat $R$-algebra.
\end{enumerate}
\end{lem}

\begin{proof}
(i) The $R$-algebra $A$, being a quotient of $B$, is finitely 
generated over $R$.

(ii) $A$ is a direct summand of the faithfully flat $R$-algebra $B$. So $A$ is  flat over $R$. 
Since $R \subseteq A \subseteq B$, and $B$ is faithfully flat over $R$, it follows that $A$ is faithfully flat 
over $R$.
\end{proof}

The next result, pertaining to the patching conditions of an $\A^1$-patch, occurs in \cite[Lemma 2.2]{DGO}.

\begin{lem}\label{easy0}
Let $C$ be an integral domain, and let $x,y$ be non-zero elements of $C$.
Then the following statements are equivalent:
\begin{enumerate}
\item[\rm(I)] $y$ is a $(C/xC)$-regular element.
\item[\rm(II)] $C=C_x\cap C_y$.
\end{enumerate} 
 \end{lem}

The following result is crucial to the proof of our main patching result. 

\begin{lem}\label{ideal} Let $R$ be an integral domain, $A$ a flat $R$-algebra, $r,x$ non-zero elements of $R$
and $I=rR_x \cap R$. Then we have:
\begin{enumerate}
\item[\rm(i)] $IA=rA_x \cap A$.
\item[\rm(ii)] If $I$ is an invertible ideal of $R$, then $I^n=r^nR_x \cap R$ for every $n \ge 0$.
\end{enumerate}
\end{lem}

\begin{proof}
(i) Since $A$ is $R$-flat, we have by Lemma \ref{bourbaki},
$$ I\otimes_R A= (rR_x \otimes_R A) \cap (R \otimes_R A),$$
considering all three as submodules of $R_x \otimes_R A$.  Hence, 
identifying $R_x \otimes_R A$ with $A_x$ and considering 
the images of the above three modules  in $A_x$, we have 
$IA=rA_x \cap A$.

(ii) Fix $n \ge 0$. Since $I= rR_x \cap R$, we have $I^n \subseteq r^nR_x \cap R$. 
To prove the equality, we assume that $R$ is a local ring, so that $I$ becomes principal.  
Let $I= c R$ for some $c\in R$. 
Then $cR_x=I_x=rR_x$ and $c^n R_x=I^nR_x=r^n R_x$. Since $I=I_x \cap R$  by construction, we have
 $cR= cR_x \cap R$ which shows that $x$ is $(R/cR)$-regular and hence $(R/c^nR)$-regular. Therefore, 
 $c^nR=c^nR_x \cap R$
 and hence $I^n =c^nR= c^nR_x \cap R=r^n R_x \cap R$. 
\end{proof}

\begin{lem}\label{patch}
Let $R$ be a Noetherian domain and  $B$ be  a faithfully flat $R$-algebra 
such that $R$ is factorially closed in $B$. 
Let $A$ be an $R$-subalgebra of $B$ such that $A$ is an $R$-algebra retract of $B$. 
Suppose there exist non-zero elements $x, y \in R$ such that
\begin{enumerate}
\item[\rm(i)] $y$ is an ($R/xR$)-regular element. 
 \item[\rm(ii)] $B_x={A_x}^{[1]}$.
 \item[\rm(iii)] $B_y={A_y}^{[1]}$.
\end{enumerate}
Then $B \cong {\rm Sym}_A(IA)$ for some invertible ideal $I$ of $R$. 
In particular, $B$ is faithfully flat over $A$. 
\end{lem}
\begin{proof}
Let $F, G \in B$ be such that $B_x= A_x[F]$ and $B_y=A_y[G]$. 
Let $\pi: B \to A$ be an $R$-algebra retraction.
Replacing $F$ and $G$ by $F-\pi(F)$ and $G-\pi(G)$, respectively,
we may assume that $\pi(F)=\pi(G)=0$. Now 
$$
B_{xy}=A_{xy}[F]=A_{xy}[G]
$$
and hence $F= \lambda G+ \mu$, for some $\lambda \in {A_{xy}}^*$
and $\mu \in A_{xy}$.   Since $\pi(F)=\pi(G)=0$, 
considering the extended retraction $\pi: B_{xy} \rightarrow A_{xy}$,
we have $\mu=\pi (\mu)$ (as $\mu \in A_{xy}$) $=0$.
Further, since $F \in B \subseteq B_y=A_y[G]$, we have $\lambda \in A_y\cap {A_{xy}}^*$.
Let $\lambda = a/y^m$ for some $a \in A$. 
Note that as $a$ is a unit in $A_{xy}$, there exists $b \in A$ such that
$ab= (xy)^n$ for some integer $n \ge 0$. Since $R$ is factorially closed in $B$, it follows that $a, b \in R\cap {R_{xy}}^*$. 
Thus $\lambda = a/y^m \in R_y \cap {R_{xy}}^*$.

Now $B$ is faithfully flat over $R$ and hence $A$ is faithfully flat over $R$ by Lemma \ref{ff}(ii). 
Since $y$ is  $R/xR$-regular and both $A, B$ are $R$-flat, $y$ is also $A/xA$-regular and $B/xB$-regular. 
Therefore, by Lemma \ref{easy0},    
$$R= R_x \cap R_y, ~~ A= A_x \cap A_y {\rm ~~ and ~~} B= B_x \cap B_y.$$ 
Let $T= G/y^m$. Then $F= aT$ and  
$$
B= B_x \cap B_y=A_x[F] \cap A_y[G]= A_x[aT]\cap A_y[T] \subseteq A_{xy}[T].
$$
Hence 
\begin{equation}\label{B}
B= \left( \bigoplus_{n \ge 0} A_xa^nT^n\right) \bigcap \left(\bigoplus_{n \ge 0} A_yT^n\right)=\bigoplus_{n \ge 0}M_n T^n, 
\end{equation}
where $M_n= a^nA_x\cap A_y$ for every $n \ge 0$.  Note that $M_n \subseteq A_x \cap A_y=A$; thus $M_n=a^nA_x \cap A$
and $M_n$ is an ideal of $A$, for each $n \ge 0$.

Set $I:= aR_x \cap R_y$. 
Then $I$ ($\subseteq R_x \cap R_y=R$) is an ideal of $R$ and by Lemma \ref{ideal}(i), $M_1=IA$.

Now, as $B$ is a flat $R$-algebra and $M_1 (\cong M_1T)$ is isomorphic to a
direct summand of $B$, we see that $M_1$ is a flat $R$-module. Hence, as $A$ is $R$-flat,
$I\otimes_R A$ is isomorphic to $IA$ ($=M_1$), and hence $I\otimes_R A$ is a flat $R$-module. 
Since $A$ is a faithfully flat $R$-algebra, 
it follows that the ideal $I$ is a flat $R$-module. Since $R$ is Noetherian, it follows that $I$ is an invertible ideal of $R$.
Therefore, by Lemma \ref{ideal} (ii), $I^n = a^n R_x \cap R$ and hence by Lemma \ref{ideal}(i), 
 $M_n \cong (a^nR_x \cap R)A=I^nA$.

Now by (\ref{B}), we have $B=\bigoplus I^nAT^n\cong{\rm Sym}_A(IA)$ for the invertible ideal $I$ of $R$. 
\end{proof}

\section{Retracts and rings of invariants of  ${\mathbb G}_a$-actions}

In this section we shall mainly study the ring of invariants $A$ of an $R$-algebra exponential map of  
$B=R^{[n]}$ when $A$ occurs as a retract of $B$, especially the relationship between $B$ and $A$,
and associated results. When $R$ contains $\mathbb Q$, we get corresponding results for the kernel $A$ of 
a locally nilpotent $R$-derivation of $B=R^{[n]}$ when $A$ occurs as a retract of $B$.
  
We first record an elementary result on symmetric algebras:
\begin{lem}\label{sym}
 Let $A$ be an integral domain and $B$ an $A$-algebra such that
 $B\cong {\rm Sym}_A(Q)$ for some invertible ideal $Q$ of $A$.
Then $A$ is a retract of $B$ and  $A=B^{\phi}$ for some $\phi(\neq Id) \in \Exp_A(B)$.
 \end{lem}

\begin{proof}
Clearly $A$ is a retract of $B$. We now show that $A= B^{\phi}$ for some $\phi(\neq Id) \in \Exp_A(B)$.
Let $K$ denote the field of fractions of $A$ and $S=A\setminus\{0\}$. Then 
$S^{-1}Q=K$ and hence $$S^{-1}B \cong {\rm Sym}_K(K)=K^{[1]}= (S^{-1}Q)^{[1]}=(S^{-1}A)^{[1]}.$$
Therefore,
since $Q$ is a finitely generated $A$-module,  there exists
an element $a \in S$ such that $$B_a= {A_a}[F]$$ for some $F\in B$ which is transcendental over $A$. 

Let $x_1, \dots, x_n \in B$ generate $B$ as an algebra over $A$. Then
there exists an integer $m \ge 0$ such that $a^m x_i \in A[F]\subseteq B$ for every $i$, $1\le i\le n$.

Now, the $A_a$-algebra homomorphism $\phi':B_a (=A_a[F]) \to B_a[U]$ defined by 
$$\phi'(F)=F+a^mU$$
is clearly a non-trivial exponential map
satisfying $\phi'(B) \subseteq B[U]$.
Hence $\phi:=\phi'|_B$ is a nontrivial exponential map of $B$. 
Since $A \subseteq B^{\phi}$, $A$ is algebraically closed in $B$ and $\td_A B=1=\td_{B^{\phi}}B$ (by Lemma \ref{lnd} (iii)), we conclude that $A= B^{\phi}$. 
\end{proof}

The next result shows that if $\phi$ is a  nontrivial exponential map of a UFD $B$, 
and if there exists a retraction from $B$ to $A:=B^{\phi}$, then $B=A^{[1]}$.

\begin{prop}\label{ufdp}
Let $B$ be a UFD  and $A$ a subring of $B$.
Then the following statements are equivalent:
\begin{enumerate}
 \item [\rm (I)] $A$ is a retract of $B$ and $A=B^{\phi}$ for some $\phi(\neq Id) \in \Exp_A(B)$.
 \item [\rm (II)] $B=A^{[1]}$.
\end{enumerate}
In particular, if $R$ is a UFD  and $A$ an $R$-subalgebra of the polynomial ring $B=R^{[n]}$
satisfying ({\rm I}), then $B= A^{[1]}$. 
\end{prop}
\begin{proof}
${\rm (I)} \Rightarrow {\rm (II)}$: 
By Lemma \ref{lnd}(i) (also by Lemma \ref{ct}(iv)),  $A$ is a UFD. Since $A$ is factorially closed in $B$
(Lemma \ref{lnd}(i)), any prime element $p$ of $A$ remains a prime element in $B$. Moreover, since $B$ is a UFD, 
$\bigcap_{n\geqslant0} p^n B=(0)$. 

By Lemma \ref{lnd}(iii), there exists a non-zero element 
$a \in A$ such that $B_{a}={A_{a}}^{[1]}$. If $a \in A^{*}$, then $B=A^{[1]}$. Otherwise, 
let $a= p_1^{m_1}\cdots p_{\ell}^{m_\ell}$ be a prime factorization
of $a$ in $A$. Then, each $p_i$ is a prime element in $B$ and $\bigcap_{n\geqslant0} {p_i}^n B=(0)$ for $1\leqslant i\leqslant \ell$. 
Therefore, by repeated application of Theorem \ref{rsdd}, $B=A^{[1]}$.

\smallskip

${\rm (II)} \Rightarrow {\rm (I)}$: Clearly, $A$ is a retract of $B$. Now if $B=A[T]=A^{[1]}$, then the $A$-algebra homomorphism
$\phi: B \rightarrow B[U]$ defined by $\phi(T)= T+U$ is a nontrivial exponential map of $B$ with $A=B^{\phi}$.
\end{proof}

\begin{cor}\label{ufdpc}
Let $B$ be a UFD containing $\bQ$ and $A$ a subring of $B$.
Then the following statements are equivalent:
\begin{enumerate}
 \item [\rm (I)] $A$ is a retract of $B$ and $A=Ker ~ D$ for some $D(\neq 0) \in LND(B)$.
 \item [\rm (II)] $B=A^{[1]}$.
\end{enumerate}
In particular, if $R$ is a UFD containing $\bQ$ and $A$ an $R$-subalgebra of the polynomial ring $B=R^{[n]}$
satisfying (\rm I), then $B= A^{[1]}$. 
\end{cor}

The following example shows that in Proposition \ref{ufdp} or Corollary \ref{ufdpc}, we cannot relax the hypothesis that $B$ is a UFD.

\begin{ex}
{\em  
Consider the ring $$B:= \bC[X,Y,Z]/(XY-Z^2),$$ a Noetherian normal domain. Let $x$, $y$ and $z$ denote respectively
the images of $X$, $Y$ and $Z$ in $B$. Then $A:= \bC[x]$ is a retract of $B$ under the retraction map
$\pi:B \to A$ defined by $$\pi(y)=\pi(z):=0.$$ Also, the $\bC$-linear derivation $D:B\to B$ defined by
$$
D(x):=0,~~ D(z):=x {\rm ~~ and ~~ } D(y):=2z
$$ is a locally nilpotent derivation whose kernel is $A$. However $B\neq A^{[1]}$
since $B\neq \bC^{[2]}$. 
}
\end{ex}

We now prove our main theorem. 

\begin{thm}\label{rlnd1}
Let $R$ be a Noetherian normal domain and $A$ an $R$-subalgebra of the polynomial ring $B:=R^{[n]}$.
Then the following statements are equivalent:
\begin{enumerate}
 \item [\rm(I)] $A$ is a retract of $B$ and $A=B^{\phi}$ for some $\phi(\neq Id) \in \Exp_A(B)$.
\item[\rm(II)] $B\cong {\rm Sym}_A(IA)$ for some invertible ideal $I$ of $R$. 
\end{enumerate}
\end{thm}

\begin{proof}
${\rm (II)} \Rightarrow {\rm (I)}$: Follows from  Lemma \ref{sym}.

\smallskip

${\rm (I)} \Rightarrow {\rm (II)}$: 
Let $\p$ be a  height one prime ideal of $R$. Then $R_{\p}$ is a discrete valuation ring and hence a UFD. 
By Lemma \ref{lnd}, $\phi$ induces a nontrivial exponential map $\phi_\p$ on $B_\p$ with the ring of invariants
$A_\p$; and by Lemma \ref{ct}(ii), 
$A_\p$ is a retract of $B_\p$. Therefore, by Proposition \ref{ufdp}, $B_\p= {A_\p}^{[1]}$. 
Since $B$ is a finitely generated $R$-algebra, 
there exists an element $x\in R\setminus \p$ such that $B_x= {A_x}^{[1]}$.

Set $\Delta:={\rm Ass}_R(R/xR)$ and $S:= R \setminus \bigcup_{\p\in U}\p$. 
Since $R$ is a Noetherian normal domain, $\hgt \p=1$ for each $\p \in \Delta$. Hence
$S^{-1}R$ is a semilocal Dedekind domain and therefore a PID. 
By Lemma \ref{ct}(ii), $S^{-1}A$ is a retract of $S^{-1}B$. Also, by Lemma \ref{lnd}(ii), 
$\phi$ induces a nontrivial exponential map $S^{-1}\phi$ of $S^{-1}B$ with ring of invariants $S^{-1}A$. 
Therefore, by again applying Proposition \ref{ufdp}, we get $S^{-1}B=(S^{-1}A)^{[1]}$. 
Since $B$ is a finitely generated $R$-algebra, there exists an element $y \in S$ such that $B_y={A_y}^{[1]}$. 

Since $y$ is $(R/xR)$-regular, by Lemma \ref{patch}, 
we have $B\cong {\rm Sym}_A(IA)$ for some invertible ideal $I$ of $R$. 
\end{proof}

\begin{cor}\label{rlnd1c}
Let $R$ be a Noetherian normal domain containing $\bQ$ and $A$ an $R$-subalgebra of the polynomial ring $B:=R^{[n]}$.
Then the following statements are equivalent:
\begin{enumerate}
 \item [\rm(I)] $A$ is a retract of $B$ and $A=Ker ~ D$ for some $D(\neq 0) \in LND_{R}(B)$.
\item [\rm(II)] $B\cong {\rm Sym}_A(IA)$ for some invertible ideal $I$ of $R$. 
\end{enumerate}
\end{cor}

The following is an explicit example of the situation 
where $R$ is a Dedekind domain,  $B=R[X,Y]$ and $A$ is an $R$-subalgebra retract of $B$ such that
$B \cong {\rm Sym}_A(IA)$ and $A \cong {\rm Sym}_R(J)$ for some non-principal invertible ideals $I$ and $J$ of $R$.
In Corollary \ref{ctd1cor2}, we shall see that any retract $A$ of $B$ satisfying $\td_RA=1$ will be the symmetric algebra of an invertible
ideal of $R$.

\begin{ex}\label{rlndex1}
{\em
Let $\bR[\alpha, \beta]=\bR^{[2]}$, 
$$
R:=\frac{\bR[\alpha, \beta]}{(\alpha^2 + \beta^2-1)}
$$ and $a, b$ denote, respectively,  the
images of $\alpha, \beta$ in $R$. Let $B:=R[X,Y]=R^{[2]}$ and $D \in LND_{R}(B)$ be defined by
$$D(X):=a {\rm~~~~and~~~~} D(Y):=b-1.$$ Set $A := Ker ~ D$, 
$$
u:= aY + (1-b)X ~ \text{and} ~ v:= (1+b)Y+aX.
$$ 
Let 
$$
F:= aY-(1+b)X  ~ \text{and} ~ G:= (1-b)Y-aX \in B.
$$
Then $u,v \in A$, 
$$B_{(1+b)}=R_{(1+b)}[v,F] {\rm~~~~and~~~~} B_{(1-b)}=R_{(1-b)}[u,G].$$ It follows that 
$$A_{(1+b)}=R_{(1+b)}[v] {\rm~~~~and~~~~} A_{(1-b)}=R_{(1-b)}[u].$$  Also
$$B_{(1+b)}=A_{(1+b)}[F] {\rm~~~~and~~~~} B_{(1-b)}=A_{(1-b)}[G].$$
Since $(1+b)$ and $(1-b)$ are comaximal ideals of $R$, we have
$$A=R[u,v]={\rm Sym}_{R}(J) {\rm~~~~where~~~~} J=(a,1-b)R$$
and 
$$B= A[F, G]= {\rm Sym}_{A}(IA), {\rm~~~~where~~~~} I=(a,1+b)R.$$ 
If $ \pi : B \rightarrow A$ is defined by
$$ \pi(X):=\frac{u}{2} {\rm~~~~and~~~~}  \pi(Y):=\frac{v}{2},$$ then $\pi$ is a retraction. 
Since $I$ and $J$ are not principal, $A \neq R^{[1]}$ and $B \neq A^{[1]}$ by Lemma \ref{eh}.  
}
\end{ex}

\begin{rem}
{\em Let ${\mathbb Q} \subseteq k\subseteq A \subseteq B= k^{[n]}$. 
Suppose $A$ is the kernel of a non-zero locally nilpotent derivation on $B$.
In general $B$ need not be flat over $A$. For example, consider the locally nilpotent derivation
$D: B = k[X,Y,Z,T] \to k[X,Y,Z,T]$ defined by 
$$D(X)=D(Y)=0, ~~ D(Z)=X {\rm~~~~and~~~~} D(T)=Y.$$ Then
$$A={\rm Ker}~D= k[X,Y, XT-YZ].$$ Clearly, $B$ is not flat over $A$ as going down fails because
$$(X,Y)B \cap A=(X,Y, XT-YZ)A \neq (X,Y)A.$$ 
In this case, $A$ is not a retract of $B$. 
Similarly, if $A$ is a retract of $B$, then $B$ need not be flat over $A$ (cf. Example \ref{nonflat}). 
However, Theorem \ref{rlnd1} shows that, if $A$ is both a retract and the kernel of a non-zero locally nilpotent derivation on $B$, 
then $B$ is faithfully flat over $A$. 
}
\end{rem}

\section{On the questions of Costa}
In this section, we discuss the following questions of D.L. Costa (\cite[Section 4]{C}).

\smallskip
\noindent
{\bf Question 1 :} What are the retracts of $R[X_1, X_2]$  when R is a normal domain, or even a Dedekind domain? 

\smallskip
\noindent
{\bf Question 2 :} What are the retracts of  $R[X_1, X_2, X_3]$ which have transcendence degree $2$ over $R$, 
where $R$ is a UFD or a field?

\smallskip
\noindent
{\bf Question 3 :} Is every retract of $k[X_1,X_2, \dots, X_n]$ a polynomial ring over $k$?

\smallskip

For convenience, we first record two well-known results on symmetric algebras and affine fibrations.

\begin{lem}\label{symretract}
Let $R$ be a ring and $A \cong {\rm Sym}_R (M)$ for a finitely generated projective $R$-module $M$.
Then $R$ is a retract of $A$ and $A$ is an $R$-algebra retract of $R^{[m]}$ for some integer $m$.    
\end{lem}

\begin{proof}
Clearly $R$ is a retract of $A$. If $M$ is generated by $m$ elements over $R$, then 
clearly $M$ is an $R$-module retract of a free $R$-module $F=R^m$ and hence $A$ is an $R$-algebra retract
of ${\rm Sym}_R (F)=R^{[m]}$.
\end{proof}

The following generalisation of the above result is an easy consequence of Asanuma's structure theorem for 
${\A}^{n}$-fibrations (\cite[Theorem 3.4]{As}).

\begin{lem}\label{fibrationretract}
Let $R$ be a Noetherian ring and $A$ an ${\A}^{n}$-fibration over $R$. Then 
$R$ is a retract of $A$ and $A$ is an $R$-algebra retract of $R^{[m]}$ for some integer $m$.
\end{lem}

\begin{proof}
By Asanuma's Theorem \cite[Theorem 3.4]{As}, $A$ is an $R$-subalgebra of a polynomial ring $B=R^{[\ell]}$ for some $\ell$ 
such that $A^{[\ell]} \cong {\rm Sym}_{B} (M)$ for a certain projective $B$-module $M$.
Since $A \hookrightarrow R^{[\ell]}$, it follows that $R$ is a retract of $A$.  By Lemma \ref{symretract},
$A^{[\ell]}$ is a retract of a polynomial ring $B^{[s]}=R^{[\ell +s]}$ for some integer $s$. Set $m:=\ell +s$.
Thus, $A^{[\ell]}$ and hence $A$ itself is a retract of $R^{[m]}$.
\end{proof}

The following result on dimension of fibre rings is probably well-known. 
For the lack of a ready reference, we present a simple proof suggested to us by N. Onoda.

\begin{lem}\label{onoda}
Let $R$ be a Noetherian domain, $B=R^{[n]}$ and $A$ a retract of $B$. Fix a prime ideal $\p$ of $R$ and set $P:=\p A$.
Then the following statements hold:
\begin{enumerate}
\item[\rm(i)] $P$ is a prime ideal of $A$, $P \cap R=\p$ and ${\rm ht}(P)={\rm ht}(\p)$.
\item[\rm(ii)] $PB \cap A=P$.
\item[\rm(iii)] $\td_{\kappa({\p})}(A \otimes_{R} \kappa({\p}))= \td_{\kappa({\p})}\kappa(P)=\td_{R}A$.
\end{enumerate}
\end{lem}

\begin{proof} (Onoda) (i) By Lemma \ref{ff}, $A$ is faithfully flat over $R$ and hence $A/{\p}A \hookrightarrow A \otimes_R \kappa({\p})$,
as $R/{\p} \hookrightarrow  \kappa({\p})$.
Since $A$ is a retract and hence a direct summand of $B$, $A \otimes_R \kappa({\p}) \hookrightarrow B \otimes_R \kappa({\p})
= {\kappa({\p})}^{[n]}$. Thus $A/P \hookrightarrow {\kappa({\p})}^{[n]}$; in particular, $P$ is a prime ideal of $A$.

The result $P \cap R= {\p}A \cap R=\p$ follows from faithful flatness of $A$ over $R$.

Since $A$ is faithfully flat over $R$, the induced map ${\rm Spec} (A) \rightarrow {\rm Spec} (R)$ is surjective (\cite[p. 28]{MA}) 
and the going-down theorem holds between $R$ and $A$ (\cite[p. 33]{MA}). Hence ${\rm ht}(P)={\rm ht}(\p A)={\rm ht}(\p)$ (\cite[p. 79]{MA}).

\smallskip

(ii) Follows from the fact that $A$ is a direct summand of $B$. 

\smallskip

(iii) Since ${\rm ht}(P)={\rm ht}(\p)$, applying the dimension inequality between $R$ and $A$ (\cite[p. 85]{MA}), we have 
\begin{equation}\label{inequality1}
\td_{\kappa({\p})}\kappa(P) \le \td_{R}A.
\end{equation}
 Note that, being a retract of the Noetherian ring $B$,
$A$ is Noetherian. Therefore, as $${\rm ht}(PB)={\rm ht}(\p B) = {\rm ht}(\p)= {\rm ht}(P),$$ applying the
dimension inequality between $A$ and $B$, we get 
\begin{equation}\label{trdeg1}
\td_{\kappa({P})}\kappa(PB) \le \td_{A}B.
\end{equation}
Now 
\begin{equation}\label{equality1}
\td_{\kappa({\p})}\kappa(P) + \td_{\kappa({P})}\kappa(PB)=\td_{\kappa({\p})}\kappa(\p B)=n
\end{equation}
and 
\begin{equation}\label{equality2}
 \td_{R}A+\td_{A}B= \td_{R}B=n.
\end{equation}
Hence, by (\ref{trdeg1}), (\ref{equality1}) and (\ref{equality2}),
$$n- \td_{\kappa({\p})}\kappa(P) \le n- \td_{R}A,$$ i.e., 
\begin{equation}\label{inequality2}
\td_{\kappa({\p})}\kappa(P) \ge \td_{R}A.
\end{equation}
 Therefore, by (\ref{inequality1}) and (\ref{inequality2}), 
$$\td_{\kappa({\p})}\kappa(P) = \td_{R}A.$$ Hence the result.

\end{proof}

We now discuss Question 1 over a general Noetherian domain $R$. Note that, by Lemma \ref{ct}(i),   
retracts of $R[X_1, X_2]$ of transcendence degree zero or two are $R$ and $R[X_1, X_2]$ respectively.
The following result shows that retracts of $R[X_1, X_2]$ of transcendence degree one are ${\A}^{1}$-fibrations;
in fact, it characterizes retracts of polynomial $R$-algebras of transcendence degree one.

\begin{thm}\label{ctd1}
Let $R$ be a Noetherian domain and $A$ an integral domain containing $R$ with $\td_{R}A=1$. Then the following statements are equivalent.
\begin{enumerate}
\item[{\rm (I)}] $A$ is an $R$-algebra retract of $R^{[m]}$ for some integer $m$.
\item[{\rm (II)}] $A$ is an ${\A}^{1}$-fibration over $R$.
\end{enumerate}
\end{thm}

\begin{proof}
\noindent 
(I) $\Rightarrow$ (II).
By Lemma \ref{ff}, $A$ is finitely generated and faithfully flat over $R$. 
 Now by Lemma \ref{ct}(ii), for each prime ideal ${\p}$ of $R$, $A \otimes_{R} \kappa({\p})$ is a
 retract of $B \otimes_{R} \kappa({\p})(={\kappa({\p})}^{[n]})$ and
by Lemma \ref{onoda}, $\td_{\kappa({\p})} (A \otimes_{R} \kappa({\p}))=1$.
 Therefore, by Theorem \ref{ct2}, 
 $A \otimes_{R} \kappa({\p})={\kappa({\p})}^{[1]}$. Hence $A$ is an ${\A}^{1}$-fibration over $R$.

\smallskip

\noindent
(II) $\Rightarrow$ (I) is a special case of Lemma \ref{fibrationretract}.
\end{proof}

\begin{rem}\label{ctd1remark1} 
{\em Since there exist non-trivial ${\A}^{1}$-fibrations over Noetherian local domains like $R={\bC}[[t^2, t^3]]$ 
(for instance, \cite[Example 2.5]{BD}), 
Theorem \ref{ctd1} is perhaps the best conclusion that one can
obtain over a general Noetherian domain $R$.    
}
\end{rem}

Recall that any ${\A}^{1}$-fibration over a Noetherian seminormal domain $R$ is  
isomorphic to a symmetric algebra of an invertible ideal of $R$ (cf. \cite[Theorem 3.10]{BD}), 
Hence, as a consequence of Proposition \ref{ctd1}, we have the following result on the precise structure
of retracts of $R^{[n]}$ of transcendence degree one, which had been shown earlier by C. Greither (\cite[Theorem 2.3]{Gr}).  

\begin{cor}\label{ctd1cor2} Let $R$ be a Noetherian seminormal domain, 
$B=R^{[n]}$ and $A$ a retract of $B$ such that $\td_{R}A=1$.
Then $A \cong {\rm Sym}_{R}(J)$ for some invertible ideal $J$ of $R$.
\end{cor}

\begin{rem}\label{ctd1remark2}
 {\em  Let $R$ be a Dedekind domain which is not a PID and let $J$ be an invertible ideal which is not principal. Since 
$J$ is generated by two elements, $J$ is an $R$-module retract of $R^2$ and hence $A = {\rm Sym}_{R}(J)$ is an $R$-algebra 
retract of $B={\rm Sym}_{R}(R^2)=R^{[2]}$. But $A \ne R^{[1]}$ by Lemma \ref{eh}. (Example \ref{rlndex1} is an explicit illustration.) 
Thus, Corollary \ref{ctd1cor2} is the best possible answer to Question 1 even for a Dedekind domain.
}
\end{rem}

The following theorem answers Question 2 affirmatively in the case when  $R$ is a field  of characteristic zero.
This result has been independently observed by T. Nagamine in \cite{T}.

\begin{thm}\label{ftd2}
 Let $k$ be a field of characteristic zero, $B=k^{[n]}$ and $A$ is a retract of $B$ with $\td_kA=2$. Then 
 $A=k^{[2]}$. In particular, any retract of $k^{[3]}$ is isomorphic to a polynomial ring over $k$.
\end{thm}

\begin{proof}
 Being a quotient of $k^{[n]}$, $A$ is an affine $k$-domain. Using Lemma \ref{e}, we may assume that $A$ can be 
 embedded in a polynomial ring $C=k^{[2]}$ as a $k$-algebra. 
By Lemma \ref{ct}(v), $A$ is regular. Let $\bar{k}$ be an algebraic closure of $k$. 
 Then $\bar k\otimes_{k} A$, being a retract 
 of $\bar k^{[n]}$, is a UFD by Lemma \ref{ct}(iv). Further, $Q(C)|_{Q(A)}$ is a 
 separable algebraic extension as $k$ has characteristic zero and $\td_{k}A=2$. Now it follows from Theorem \ref{r} 
 that $A=k^{[2]}$. 
\end{proof}

The following result addresses Question 2 when ${\rm dim~} R \ge 1$.  

\begin{thm}\label{td2n}
Let $R$ be a Noetherian domain containing $\bQ$ and $A$ an integral domain containing $R$
for which $\td_{R}A=2$. Then the following statements are equivalent.
\begin{enumerate}
\item[{\rm (I)}] $A$ is an $R$-algebra retract of $R^{[m]}$ for some integer $m$.
\item[{\rm (II)}] $A$ is an ${\A}^{2}$-fibration over $R$.
\end{enumerate}
\end{thm}

\begin{proof} 
\noindent 
(I) $\Rightarrow$ (II).
Since $A$ is a retract of $B$, by Lemma \ref{ff}, $A$ is finitely generated and faithfully flat over $R$. 
Now by Lemma \ref{ct}(ii), for each prime ideal ${\p}$ of $R$, $A \otimes_{R} \kappa({\p})$ is a retract of 
$B \otimes_{R} \kappa({\p})(={\kappa({\p})}^{[n]})$; and hence 
by Lemma \ref{onoda}, $\td_{\kappa({\p})} (A \otimes_{R} \kappa({\p}))=2$.
Therefore, by 
Theorem \ref{ftd2}, $A \otimes_{R} \kappa({\p})={\kappa ({\p})}^{[2]}$. Hence $A$ is an ${\A}^2$-fibration over $R$.

\smallskip

\noindent
(II) $\Rightarrow$ (I) is a special case of Lemma \ref{fibrationretract}. 
\end{proof}

As a consequence, we have the following response to Question 2 for a Dedekind domain $R$.
  
\begin{cor}\label{td2ncor1}
Let $R$ be a Dedekind domain containing $\bQ$, $B:=R[X_{1}, \dots ,X_{n}](=R^{[n]})$ and $A$ a retract of $B$ such that 
$\td_{R}A=2$. Then $A\cong {\rm Sym}_{R}(M)$ for some finitely generated projective $R$-module $M$
of rank two.
\end{cor} 

\begin{proof}
For each maximal ideal $\m$ of $R$, $R_{\m}$ is a discrete valuation ring and therefore, by Theorem \ref{td2n},  $A_{\m}$ is
an ${\A}^{2}$-fibration over $R_{\m}$.  Hence, 
by Theorem \ref{sth}, $A_{\m}={R_{\m}}^{[2]}$. The result now follows from  Theorem \ref{bcw}.
\end{proof}

\begin{rem}\label{q3}
{\em  
(i) In \cite[Theorem 5.1]{As}, Asanuma showed that even over a  discrete valuation ring not containing $\bQ$,
there are $\A^2$-fibrations which are not polynomial rings. Therefore, Theorem \ref{td2ncor1} seems to be the best possible result
in general.  When $R$ is a regular local ring and $A$ is an $\A^2$-fibration over $R$, 
then Asanuma has shown  that $A^{[m]} =R^{[m+2]}$ for some integer $m$ (\cite[Corollary 3.5]{As}).
For a more general statement of Asanuma on the structure of affine fibrations over Noetherian rings, see \cite[Theorem 3.4]{As}.

(ii) Let $R$ be an integral domain having a non-free projective module $M$ of rank $n$ (for instance, if $R$ has a non-trivial Picard group)  
and $A={\rm Sym}_R(M)$. Then $\td_RA =n$,  $A$ is a retract of a polynomial algebra $B$ over $R$ (by Lemma \ref{symretract})
but $A$ itself is not  a polynomial algebra over $R$ (by Lemma \ref{eh}). 
Thus, a result like Corollary \ref{td2ncor1} seems to be the best possible even over a Dedekind domain.
}
\end{rem}

We now discuss Question 3. The following remark shows that over a field of positive characteristic, 
Question 3 does not have an affirmative answer in general.

\begin{rem}\label{neena}
{\em Let $k$ be any field of positive characteristic. Using Asanuma's example of a non-trivial $\A^2$-fibration over $k^{[1]}$, 
 the fourth author has proved that there exist stably polynomial rings over $k$ (and hence retracts 
of polynomial rings over $k$) which are not themselves polynomial rings over $k$ (\cite{G}).  
In fact, using the examples of \cite{G2}, it can be shown that for any 
$n\geqslant 4$, there exist retracts $A$ of $k^{[n]}$, satisfying $3 \leqslant \td_k A\leqslant n-1$ which are not polynomial 
rings.}  
\end{rem}

In the context of Question 3, the next result gives a sufficient condition for a retract of $R[X_1, \dots ,X_n]$
to be a polynomial ring. 
Recall that for a graded ring $B=\bigoplus_{i \ge 0}B_{i}$, the ideal of $B$ generated by the 
homogeneous elements of positive degree is denoted by $B_{+}$.

\begin{thm}\label{gradedr}
Let $R$ be an integral domain and $A$ a graded $R$-subalgebra of the polynomial ring 
$B:=R[X_1, \dots ,X_n]$ with standard grading. Suppose that there exists a retraction $\pi : B\rar A$ such that $\pi(B_+) \subseteq B_+$. 
Then $A = {\rm Sym}_R(M)$ for some finitely generated projective 
$R$-submodule $M$ of the free $R$-module $B_1=RX_1\oplus RX_2\oplus \cdots\oplus RX_n$.
\end{thm}

\begin{proof} 
Since $A$ is a graded subring of $B$, we have $A_+= A \cap B_+$ and $A_1= A \cap B_1$.
 Set $F:= B_+/{B_+}^2$. The  quotient map $\eta: B_+\to F$  restricts to an $R$-linear isomorphism
$$
\theta: B_1 \to F
$$  
of free $R$-modules given by 
$$
\theta(X_i)= (X_i+ {B_+}^2)/{B_+}^2  {\rm~~for~~} 1\le i \le n.  
$$
As $A$ is a graded subring $B$, we have $A_1+{B_+}^2=A_++{B_+}^2$ and hence under the isomorphism $\theta$, we have 
\begin{equation}\label{a1}
\theta(A_1)= (A_1+{B_+}^2)/{B_+}^2=(A_++{B_+}^2)/{B_+}^2. 
\end{equation}
Since $\pi(B_+) \subseteq B_+$,  $\pi$ induces an idempotent endomorphism 
$\bar \pi$ of $F$ given by 
$$
\bar{\pi}(g~{\rm mod} {B_+}^2)= \pi(g)~ {\rm mod} {B_+}^2,
$$ 
i.e., $\bar{\pi}$ satisfies $\bar{\pi}\theta=\eta\pi$. 
Hence we have an induced idempotent endomorphism $\tilde{\pi}$ of $B_1$ given by 
$$
\tilde{\pi}= \theta^{-1}\bar{\pi}\theta. {\rm ~i.e.~}, {\theta}\tilde\pi=\bar{\pi}\theta=\eta\pi.
$$
Since $\pi(B_+) \subseteq B_+$, $A_+\subseteq \pi(B_+)$, $A=\pi(B)$ and  $A_+= A \cap B_+$, 
we have $\pi(B_+)=A_+$ and hence, by (\ref{a1}), 
$$
\bar{\pi}(F)=(A_++{B_+}^2)/{B_+}^2=\theta(A_1).
$$  
Let $M= \tilde{\pi}(B_1)$, $N= $ker$(\tilde{\pi})$ and rank$(\tilde{\pi})=$ rank$(\bar{\pi})=d$.
Since $\tilde{\pi}$ is idempotent, $M$ and $N$ are projective $R$-submodules of $B_1$ and $B_1= M \oplus N$. 
Since $\eta\pi(N)={\theta}\tilde\pi(N)=0$, we have $\pi(N) \subseteq {B_+}^2$. 
Further,
$$
M=\tilde{\pi}(B_1)= \theta^{-1}\bar{\pi}\theta(B_1)=\theta^{-1}\bar{\pi}(F)=\theta^{-1} \theta(A_1)=A_1 \subset A.
$$
Hence ${\rm Sym}_R(M) \subseteq A$.
We now prove that $A\subseteq {\rm Sym}_R(M)$. It is enough to prove the statement locally. Thus,  we assume that $R$ is a local ring
and therefore both $M$ and $N$ are free $R$-modules of rank $d$ and $n-d$ respectively, say $M= RY_1\oplus \dots\oplus RY_d$
and $N= RY_{d+1}\oplus \dots\oplus RY_{n}$.    Then, we have  $B=R[Y_1, \dots ,Y_n]$ and $R[Y_1, \dots, Y_d]\subseteq A$. 

Since $M\subseteq A$ and  $\pi(N) \subseteq {B_+}^2$, we have
\begin{equation}\label{yi}
Y_i=\pi(Y_i) {\rm~~if~~} 1\leqslant i\leqslant d ~~{\rm and~~} \pi(Y_i) \in {B_+}^2 {\rm~~if~~} d+1\leqslant i\leqslant n. 
\end{equation}
We now show  that $A\subseteq R[Y_1, \dots ,Y_d]$.
Suppose not. Then there exists a homogeneous polynomial 
$g(Y_1,\dots,Y_n) \in A \setminus R[Y_1, \dots ,Y_d]$. But then $\pi(g)\neq g$ as $\pi(Y_i) \in {B_+}^2$ for all 
$i>d$ by (\ref{yi}), a contradiction. Therefore $A=R[Y_1, \dots ,Y_d]= R^{[d]}$. 
This completes the proof.
\end{proof}

\begin{rem}
{\em
We note that Theorem \ref{gradedr} does not require the retraction map $\pi$ to be a graded homomorphism.
For example, let $A:=k[X] \subseteq B:=k[X,Y]$ be a retract with a retraction map sending $Y$ to a non-constant polynomial  
$f(X)$ which is not homogeneous and $f(0)=0$. Then $A$ is a graded subring of $B$ and 
$\pi(B_+)\subseteq B_+$. However, $\pi$ is not a graded homomorphism. 
}
\end{rem}

As a consequence of Theorem \ref{gradedr}, we have the following result over fields.

\begin{cor}\label{gradedc}
Let $A$ be graded $k$-subalgebra of the polynomial ring 
$B:=k[X_1, \dots ,X_n]$ with standard grading. Suppose that there exists a retraction 
$\pi : B\rar A$ such that $\pi(B_+) \subseteq B_{+}$. Then there exists a matrix $\sigma \in {\rm Gl}_n(k)$
such that $A=k[\sigma(X_1). \dots, \sigma(X_d)]$ for some $d \le n$. In particular, 
$A$ is isomorphic to a polynomial ring over $k$.
\end{cor}
\begin{proof}
Since projective modules over a field are free, we may assume as in the proof of Theorem \ref{gradedr}, that
$M= RY_1\oplus \dots\oplus RY_d$ and $N= RY_{d+1}\oplus \dots\oplus RY_{n}$. Thus, there exists a matrix $\sigma \in {\rm Gl}_n(k)$
such that $\sigma (X_i)=Y_i$ for $1\le i \le n$. Hence the result follows from Theorem \ref{gradedr}.
\end{proof}

The following remark summarises the status of Question 3.

\begin{rem}\label{status}
{\em Let $A$ be a retract of $k^{[n]}$. If $\td_k A=0, 1 ~{\rm or}~ n$, then by Theorem \ref{ct2}, 
$A$ is also a polynomial ring over $k$.  If $k$ is a field of characteristic zero and $\td_k A=2$,  
then it follows from Theorem \ref{ftd2} that $A$ is again a polynomial ring. 
We have also seen in Remark \ref{neena} that if $k$ is a field of positive characteristic,  $n\geqslant 4$ and $3 \leqslant \td_k A\leqslant n-1$, then $A$ need not be a polynomial ring. Question 3 therefore reduces to the following two questions.}
\end{rem}

\smallskip
\noindent
{\bf Question 3.1 :} Let $k$ be a field of characteristic zero, $n \geqslant 4$ and $A$
a retract of $k[X_1, \cdots, X_n]$ with $\td_k A \ge 3$. Does it follow that $A$ is a polynomial ring over $k$? 

\smallskip
\noindent
{\bf Question 3.2 :} Let $k$ be a field of positive characteristic,  $n \geqslant 3$ and $A$
a retract of $k[X_1, \cdots, X_n]$ with $\td_k A =2$. Does it follow that $A$ is a polynomial ring over $k$? 

\smallskip

In the context of Question 3, we ask the following weaker question.

\smallskip
\noindent
{\bf Question 4 :} Suppose $A$ is a retract of $k^{[n]}$. 
Does it follow that the field of fractions of $A$ is a purely transcendental extension of $k$? 

\begin{rem} 
{\em Question 4 has an affirmative answer whenever Question 3 has an affirmative answer. 
Over a field of positive characteristic, 
the counterexamples to the Zariski Cancellation Problem (\cite{G} and \cite{G2}) are counterexamples to Question 3 for each $n \geqslant 4$.  
However, in each of the examples in \cite{G} and \cite{G2}, the 
field of fractions are purely transcendental extensions of $k$ and thus are not counterexamples to Question 4.
Therefore, Question 4 is open for any field and any integer $n \geqslant3$.
}
\end{rem}

\section{Miscellaneous results on retracts}

Let $A$ be a ring and $B:=A[X_{1}, X_{2}, \dots, X_{n}]$.
Then $A$ is (trivially) a retract  of $B$ under the natural retraction map which sends each $X_{i}$ 
to $0$. Now for any ideal $\q$ of $B$, contained in $(X_{1}, X_{2}, \dots, X_{n})B$, $A$ is also a 
retract of $B/\q$.  So even if $A$ is a very nice ring, say a polynomial ring over
a field $k$, the ring $B/\q$ can 
be `virtually anything'. Therefore, we do not consider ascent properties under a retraction. Instead, we only 
focus on some nice properties of $B$ and check whether they are preserved under a retraction.
In this section, we record a few results of this type. 
For an ideal $I$ of a ring $R$, we denote the minimal number of generators of $I$ by $\mu(I)$. 
For a local ring $R$, we use ${\m}_{R}$ to 
denote the unique maximal ideal of $R$. 

\smallskip

\begin{lem}\label{retract1} 
Let $A$ be a subring of $B$. If there exists a retraction $\pi:B \rar A$, then the following results hold.
 \begin{enumerate}
  \item [\rm (i)]If $J$ is an ideal of $A$ then $\mu(JB)=\mu(J)$.
  \item [\rm (ii)]For any ideal $\q$ of $B$ with $\q \subseteq \text{Ker }\pi$, $A$ is also a retract of $B/\q$.
  \item [\rm (iii)]Let $\p$ be an ideal of $B$ with $\text{Ker }\pi \subseteq \p$. Then $\pi(\p)=\p\cap A$. If further  
 $\p \in {\rm Spec~}B$, then $A/(\p \cap A)$ is a retract of $B/\p$ and 
 $A_{\p\cap A}$ is a retract of $B_\p$.
  \end{enumerate}
\end{lem}

\begin{proof}
(i) Since any generating set of $J$ also generates $JB$, $\mu(JB)\leqslant\mu(J)$. Similarly, $\mu(J)\leqslant\mu(JB)$ as 
$\pi(JB)=J$. Thus $\mu(JB)=\mu(J)$.

\smallskip

(ii) Since $\q \cap A \subseteq \pi(\q)=(0)$, we get an inclusion $A \subseteq B/\q$. As 
$\q \subseteq \text{Ker }\pi$, $\pi :B \rar A$ factors through $B/\q$, inducing a retraction $\bar \pi :B/\q \rar A$. 

\smallskip

(iii) Trivial.
\end{proof}

The next example shows that if $A$ is a retract of $B$, then
for an arbitrary prime ideal $\p$ of $B$, $A/(\p \cap A)$ need not be a retract of $B/\p$.

\begin{ex}
{\em
Let $A=k[X]$ and $B=k[X,Y]=A[Y]$, with the retraction map $\pi:B \to A$ being the $A$-algebra map defined by $\pi (Y)= 0$. 
Let $\p=(Y^2-X^3)B$. Then $\p \cap A=(0)$ and $B/\p=A[y]$, where $y^2=X^3$. Here $y \in B/\p$ is algebraic over $A$ but $y \not \in A$.
Hence $A$ is not algebraically closed in $B/\p$. Therefore, by Lemma \ref{retract1} (i),
$A$ cannot be a retract of $B/\p$. 
}
\end{ex}

The next example shows that if $A$ is a retract of $B$, then
for an arbitrary prime ideal $\p$ of $B$, $A_{(\p \cap A)}$ need not be a retract of $B_{\p}$.

\begin{ex}
{\em
Let $A=k[X]$ and $B=k[X,Y,Z]=A[Y,Z]$, with the retraction map 
$\pi:B \to A$ being the $A$-algebra map defined by $\pi(Y)= 0, \pi(Z)=0$.
Let $\p=YB$. Then $A_{(\p \cap A)}=k(X)$ and $B_{\p}=k[X,Y,Z]_{(Y)}$. 
Suppose, if possible, that there exists a retraction $\phi: B_\p \to A_{(\p \cap A)}$. Let $\phi(Z)=f(X)$.
Then $\phi(Z-f(X))=0$, which is not possible as $Z-f(X)$ is a unit in $B_{\p}$. Thus, 
$A_{(\p \cap A)}$ cannot be a retract of $B_{\p}$.
}
\end{ex}

The next result gives conditions under which a prime element of a Noetherian domain remains prime under retraction.

\begin{lem}\label{primee}
 Let $B$ be an integral domain satisfying ascending chain condition on principal ideals and $p$  a prime 
  element of $B$. Let $A$ be a retract of $B$ and $\pi: B \to A$ be a retraction map. 
If $\pi(p)$ is not a unit in $A$ then either $pB \cap A=(0)$ or $pB\cap A=\pi(p)A$. In particular, if 
  $pB \cap A \neq (0)$ then $\pi(p) $ is a prime element of $A$.
\end{lem}
\begin{proof}
Let $q=pB \cap A$. If $q=(0)$, then we are done. So we assume that $q \neq (0)$. 
Since $B$ satisfies ascending chain condition on principal ideals, so does $A$ and hence, as $q$ is a prime ideal, 
there exists a non-zero irreducible element $x \in q$. Now $q=\pi(q)= pB \cap A \subseteq \pi(pB)=\pi(p)A$ and 
$\pi(p)A$ is a proper ideal.   Hence, by irreducibility of $x$, we have $\pi(p)A= xA$. Therefore, $q= \pi(p)A$
and hence $\pi(p)$ is a prime element of $A$.
\end{proof}

The following example shows that the hypothesis that $\pi(p)$ is not a unit in $A$ is crucial in the above result.

\begin{ex}\label{rprimee1}
 {\em 
Let $A=k[XY,XZ]$, $B=k[X,Y,Z]$
and let $\pi: B \to A$ be the retraction map defined by $\pi(X):=1,\pi(Y):=XY$ and $\pi(Z):=XZ$. Then $X\in B$ is a prime 
element of $B$, but $XB \cap A=(XY,XZ)A$ has height two. Note that here $\pi(X)$ is a unit in $A$.
} 
\end{ex}

The next example shows that in Lemma \ref{primee}, even if $\pi(p)$ is not a unit in $A$, it is possible that 
 $\pi(p)A \neq pB \cap A$ and $\pi (p)$ is not prime in $A$.

\begin{ex}\label{rprimee2}
 {\em 
 Let $A=k[X]_{(X)}$, $B:=k[X,Y]_{(X,Y)}$ and 
 $\pi: B \to A$ be the retraction map defined by $\pi(Y)=0$. Then $Y+X^2$ 
 is a prime element of $B$ whereas $\pi(Y+X^2)=X^2$ is not a prime element of $A$. Here $(Y+X^2)B \cap A=(0)$. 
Thus in general, $\pi(p)A \neq pB \cap A$ for  a prime element $p$ of $B$.

} 
\end{ex}

We shall now describe a few situations where a ring $B$ is faithully flat over its retract $A$. The first result below,
an analogue of Theorem \ref{gradedr}, shows that the retract of a power series ring over  a field $k$ is always a power series ring.

\begin{thm}\label{complocalt}
Let $B:=k[[X_1, \dots ,X_n]]$ be a power series ring in $n$ indeterminates over $k$ and let 
$A$ be a subring of $B$  with a retraction map $\pi:B \rar A$. Then there exists a set of indeterminates $Y_1, \dots,Y_d,\dots,Y_n \in B$ such that $B=k[[Y_1, \dots ,Y_n]]$, 
 $A= k[[Y_1, \dots ,Y_d]]$ and $\pi(Y_i)=0$ for all $i>d$. In particular, $B$ is faithfully flat over $A$. 
\end{thm}

\begin{proof} 
By Lemma \ref{ct}, $A$ is an equicharacteristic complete regular local ring with a residue field $k$. 
Let $\m= (X_1, \dots, X_n)$ be the maximal ideal of $B$ and $\n= \pi(\m)$. Then $\n$ is the maximal ideal of $A$.  
Since $B$ is a local ring, the retraction map $\pi$ induces an idempotent 
endomorphism, say $\bar \pi$, of the $n$-dimensional $k$-vector space 
$\m/\m^2$ with $\bar{\pi}(\m)=\n/(\n \cap \m^2)= (\n+\m^2)/\m^2$. Let $d=$ rank $\bar \pi$.
Since $\bar{\pi}$ is an idempotent endomorphism, there exist $Z_1, \dots, Z_n \in \m$ such that 
$\m=(Z_1, \dots, Z_n)$, 
$$
Z_i-\pi(Z_i) \in \m^2 {\rm ~if~} 1\leqslant i\leqslant d {\rm ~ and~} \pi(Z_i)\in \m^2 {\rm ~if~} d+1\leqslant i\leqslant n.
$$
Since $B$ is a complete local ring, we then have $B=k[[Z_1, \dots, Z_n]]$ (cf. \cite[proof of Theorem 29.4]{M}).  
Let $Y_1, \dots ,Y_n \in B$ be defined by 
$$
Y_i:=\pi(Z_i) {\rm ~if~} 1\leqslant i\leqslant d {\rm ~ and~} Y_i:=Z_i-\pi(Z_i) {\rm ~if~} d+1\leqslant i\leqslant n.
$$
Then  $\pi(Y_i)= Y_i$ for $1\leqslant i\leqslant d $ and $\pi(Y_i)=0$ for $d+1 \leqslant i\leqslant n$ and  
$(Z_1, \dots, Z_n)+\m^2=(Y_1, \dots, Y_n)+\m^2$. Hence $\m= (Y_1, \dots, Y_n)$ and $\n=\pi(\m)=(Y_1, \dots, Y_d)$.
Hence, as both $B$ and $A$ are complete
regular local rings, we have  $B=k[[Y_1,\dots ,Y_n]]$ and $A=k[[Y_1, \dots ,Y_d]]$. 
\end{proof}

In the above theorem, $B$ is a complete regular equicharacteristic Noetherian local ring. The next result shows that
faithful flatness is preserved even when $B$ is not complete.

\begin{prop}\label{regular}
Let $B$ be an equicharacteristic Noetherian regular local ring and $A$ a retract of $B$. Then $B$ is faithfully flat over 
$A$.
\end{prop}

\begin{proof}
 Let $\iota:A \rar B$ and $\pi:B \rar A$ be the natural inclusion and the retraction map respectively. Since $A$ is a quotient ring of $B$,
$A$ is a local ring and $A/{\m}_{A}=B/{\m}_{B}$. If $\hat A$  and $\hat B$ denote the completion of $A$ and $B$ respectively, 
 then we have induced maps of complete local rings $\hat \iota:\hat A \rar \hat B$ and 
 $\hat \pi : \hat B \rar \hat A$. 
 \[
\begin{tikzcd}
A \arrow[r, "\iota"]\arrow[d] & B \arrow[r, "\pi"] \arrow[d] & A \arrow[d] \\
 \hat{A}\arrow[r, "\hat{\iota}"]&\hat{B}\arrow[r,"\hat{\pi}"] & \hat{A}
\end{tikzcd}
 \]
 Since $\pi\circ \iota =id_A$, $\hat \pi \circ \hat \iota=id_{\hat A}$ which implies that 
 ${\hat \iota} :\hat A \rar \hat B$ is injective and $\hat A \subseteq \hat B$ is a retract. Since $\hat B$ is an equicharacteristic 
 complete regular local ring, $\hat B\cong L^{[[n]]}$ where $L=B/{\m}_{B}$ and $n:=\text{dim }B$. 
 Therefore by Theorem \ref{complocalt},  $\hat B$ is faithfully flat over $\hat A$.
 Since $\hat A$ is faithfully flat over $A$, it follows that $\hat B$ is faithfully flat over $A$.   Since $\hat B$ is faithfully flat over both $A$ and $B$, it follows that $B$ is faithfully flat over $A$.
\end{proof}

\begin{cor}\label{ffl}
 Let $A$ be a subring of a polynomial ring $B:=k^{[n]}$ for which there exists a retraction $\pi:B\rar A$. Then, 
 for any maximal ideal $\m$ of $ B$ containing  $\text{ker }\pi$,  $B_\m$ is  
 faithfully flat over $A_{\m\cap A}$.
\end{cor}

\begin{proof}
Let $\m$ be a maximal ideal of $B$ such that $\text{ker }\pi \subseteq \m$. Then by Lemma \ref{retract1} (iii),
$A_{\m\cap A}$  is also a retract of  $B_\m$ under the induced map. The result now follows from Proposition \ref{regular}.
\end{proof}

The following example shows that, in general, a polynomial ring $B=k^{[n]}$  need not be faithfully flat over its retract, 
that the hypothesis that $\text{ker }\pi \subseteq \m$ is necessary in Corollary \ref{ffl}, and that the going-down property
may not hold for an extension $A \subset B$ with $A$ being a retract of $B$, even when $B=k^{[n]}$.

\begin{ex}\label{nonflat}
{\em  
Consider the polynomial ring $B=k[X,Y,Z]$ and its subring $A=k[XY,XZ]$. Then $A$ is a retract of $B$ with the retraction map 
$\pi : B \rar A$ defined by 
$\pi(X):=1$, $\pi(Y):=XY$ and $\pi(Z):=XZ$. But $B$ is not faithfully flat over $A$ as $XB\cap A= (XY, XZ)A$ leads to the failure 
of the going-down property. 
 
For ${\m}=(X,Y,Z)B$ and $\n= \m \cap A$, we see that $\text{ker }\pi \nsubseteq \m$ and $A_{\n}$ is not a retract of $B_\m$.
 
}
\end{ex}

In the above example, $\td_{k}A=2$.  The next result shows that if $B=R^{[n]}$ over a Noetherian domain $R$ and if $A$ is
a retract of $B$ with $\td_{R}A=1$, then $B$ is faithfully flat over its retract $A$.

\begin{thm}\label{fftd1}
Let $R$ be a Noetherian domain, $B=R^{[n]}$ and $A$ a retract of $B$ such that $\td_{R}A=1$.
Then $B$ is faithfully flat over $A$.
\end{thm}
\begin{proof}
Since $A$ is a retract of $B$, the induced map $\text{Spec }B \rightarrow \text{Spec }A$ is surjective. So all we need is to 
show that $B$ is flat over $A$. By (\cite[3.J]{MA}), it is enough to prove that $B$ is locally flat over $A$.
Let $\q$ be a prime ideal of $B$, $\p':=\q \cap A$ and $\p:=\p'\cap R$. Then we get the following local homomorphisms: 
$$
 R_{\p} \longrightarrow A_{\p'} \longrightarrow B_{\q} .
$$
By (i), $A \otimes_{R} \kappa(\p)={\kappa(\p)}^{[1]}$, i.e., a PID. Since $B \otimes_{R} \kappa(\p)$ is a 
torsion-free module over the PID $A \otimes_{R} \kappa(\p)$, it is flat over 
$A \otimes_{R} \kappa(\p)$. 
Therefore, it follows that $B_{\q} \otimes_{R_{\p}} \kappa(\p)$ is flat over $A_{\p'} \otimes_{R_{\p}} \kappa(\p)$. 
Since $B=R^{[n]}$, $B_{\q}$ is flat over $R_{\p}$. Therefore, by Lemma \ref{akl}, 
$B_{\q}$ is flat over $A_{\p'}$. 
\end{proof}

Example \ref{nonflat} shows that going-down property may fail under a retraction.
It is easy to see that going-up property too may fail under a retraction from $B$ to $A$, 
even when $B=k^{[n]}$, as the following well-known example shows (cf. \cite[p. 37]{MA}).

\begin{ex}
{\em
Let $A=k[X]$, $ B=k[X, Y]$  and $\pi: B \to A$ be the retraction map sending $Y$ to $0$. Then the prime
ideal $\p:=(XY-1)B$ contracts to $(0)$ in $A$ but there does not exist any prime ideal $\q$ of $B$ containing $\p$ which lies over 
the prime ideal $XA$. 
}
\end{ex}

If $A$ is the kernel of a locally nilpotent derivation on $B$, then any field $L$ contained in $B$ is also contained in $A$.
However, the following example shows that, in general, if $A$ is a retract of $B$ and $B$ contains a field $L$ then $L$ may not be contained in $A$.

\begin{ex}\label{rf}
{\em 
Let $A=k(X)$ and $B:=k[X,Y]_{(Y)}$. Then $A$ is a retract of $B$ under the retraction map sending $Y$ to $0$. 
The field $L=k(X+Y)$ is contained in $B$, but not in $A$.
}
\end{ex}

The following lemma gives a criterion for a retract $A$ of a ring $B$ to contain every field which is contained in $B$. 

\begin{lem}\label{retract2}
 Let $A$ be a subring of a commutative ring $B$ and let $\pi:B\rar A$ be a retraction. If the set of fields contained 
 in $B$ forms a directed set under set inclusion, i.e., for any two fields $F_1,F_2 \subseteq B$ there exists another field 
 $E \subseteq B$ such that $F_1\cup F_2 \subseteq E$, then $B$ contains a largest field, say $K$, and $K \subseteq A$. 
 In particular, if $k$ is a field and $A$ is a retract of $k^{[n]}$, then $k \subseteq A$.
\end{lem}

\begin{proof}
 The first assertion that $B$ contains a largest field $K$ follows from {\it Zorn's lemma}. Let $L:=K \cap A$. Then $L$ 
 is a field since any non-unit in $A$ remains a non-unit in $B$ (cf. Lemma \ref{ct}(i)). 
 
 If possible suppose
 $L \neq K$ and let $t \in K \setminus L$ and  $s=\pi(t)$. Then $\pi$ induces an isomorphism of the fields $L(s)$ and $L(t)$, 
 where $L(s) \subseteq A$. 
 By our hypothesis, there exists a field  $E\subseteq B$ containing both $L(s)$ and $L(t)$.  Note that $s-t \neq 0$ but $\pi(s-t) =0$. 
 This is a contradiction since $s-t$ is a unit in $E$ and therefore also in $B$. 
\end{proof}

We have seen (Lemma \ref{ct}) that the property of being a UFD or a regular ring or a normal domain is preserved under retractions. 
It is also easy to see that a retract $A$ of a seminormal domain $B$ is also a seminormal domain
as $A= B \cap Q(A)$. However, the following example shows that a retract of a Cohen-Macaulay ring may not be Cohen-Macaulay. 
\begin{ex}
{\em Let 
\[
A:=\left(\dfrac{k[X,Y]}{(X^2,XY)}\right)_{(X,Y)} \text{~and~} B:=\left(\dfrac{k[X,Y,Z]}{(X^2,XY,YZ)}\right)_{(X,Y,Z)}.
\] 
Let $\pi: B\to A$ be the retraction map defined by $\pi(z)=0$. Then $B$ is a Cohen-Macaulay 
ring since $B$ is one-dimensional and $y+z$ is a $B$-regular element. However, $A$ is not a Cohen-Macaulay ring
as $\dim A=1$, but ${\rm depth}~ A=0$.
}
\end{ex}

The following example shows that a retract of a Gorenstein ring may not be Gorenstein.

\begin{ex}
{\em  
Let
 \[A:=\dfrac{k[X,Y]}{(X^2,Y^2,XY)} \text{~and ~} B:=\dfrac{A[Z,W]}{(Z^2,W^2,ZW,xW,yZ,xZ-yW)}.\]
 Let $\pi:B \to A$  be a retraction map defined by  $\pi(z)=\pi(w)=0$. Then the vector space dimension of $A$ (over $k$) is $3$ 
 and that of $B$ is $6$. Now $A$ is not a Gorenstein ring 
since $\text{ann}_A~x=\text{ ann}_A~y=\m_A$. However, one can check that $xz$ is the only element in $B$, up to units, whose 
annihilator is equal to $\m_B$, implying that $B$ is Gorenstein.
}
\end{ex}

{\bf Acknowledgements.} 
The fourth author acknowledges Department of Science and Technology for their SwarnaJayanti Fellowship.

\bibliographystyle{amsplain}

\end{document}